\documentclass[bj,preprint]{imsart}
\RequirePackage{amsthm,amsmath,amsfonts,amssymb}
\RequirePackage[numbers]{natbib}
\RequirePackage[colorlinks,citecolor=blue,urlcolor=blue]{hyperref}
\RequirePackage{graphicx}

\usepackage[utf8]{inputenc}
\usepackage{mathrsfs}
\usepackage{mathtools}
\usepackage{nicefrac}

\startlocaldefs

\newtheorem{theorem}{Theorem}[section]
\newtheorem{lemma}[theorem]{Lemma}

\newtheorem{corollary}[theorem]{Corollary}
\newtheorem{conjecture}{Conjecture}
\theoremstyle{remark}
\newtheorem{remark}[theorem]{Remark}
\newtheorem{definition}[theorem]{Definition}
\newtheorem{example}[theorem]{Example}

\endlocaldefs

\newcommand{\E}{\mathbb{E}}
\newcommand{\Prob}{\mathbb{P}}
\newcommand{\R}{\mathbb{R}}
\newcommand{\N}{\mathbb{N}}
\newcommand{\Z}{\mathbb{Z}}
\newcommand{\I}{\mathbb{I}}
\newcommand{\bs}{\boldsymbol}
\newcommand{\dd}{\mathrm{d}}
\newcommand{\ee}{\mathrm{e}}

\newcommand{\abs}[1]{\left\vert{#1}\right\vert}
\newcommand{\defin}[1]{\textbf{#1}}
\DeclareMathOperator*{\supp}{supp}
\DeclareMathOperator*{\Var}{Var}

\newcommand{\given}{\,\vert\, }

\begin{document}

\begin{frontmatter}
\title{A characterization of the strong law of large numbers for Bernoulli sequences}
\runtitle{A characterization of the strong LLN for Bernoulli sequences}

\begin{aug}
\author[A]{\fnms{Luísa} \snm{Borsato}\ead[label=e1]{luisabborsato@gmail.com}},
\author[B]{\fnms{Eduardo} \snm{Horta}\ead[label=e2,mark]{eduardo.horta@ufrgs.br}}
\and
\author[B]{\fnms{Rafael} \snm{Rigão Souza}\ead[label=e3,mark]{rafars@mat.ufrgs.br}}
\address[A]{Institute of Mathematics and Statistics,
Universidade de São Paulo, São Paulo, Brazil.
\printead{e1}}

\address[B]{Institute of Mathematics and Statistics,
Universidade Federal do Rio Grande do Sul, Porto Alegre, Brazil.
\printead{e2,e3}}
\end{aug}

\begin{abstract}
The law of large numbers is one of the most fundamental results in Probability Theory. In the case of independent sequences, there are some known characterizations; for instance, in the independent and identically distributed setting it is known that the law of large numbers is equivalent to integrability. In the case of dependent sequences, there are no known general characterizations --- to the best of our knowledge. We provide such a characterization for identically distributed Bernoulli sequences in terms of a product disintegration.
\end{abstract}


\end{frontmatter}

\section{Introduction}
It is somewhat intuitive to most people\footnote{One may argue that most people interpret probability --- at least when it comes to coin-throwing --- in a Popperian sense, i.e. seeing probability statements as utterances which quantify the \emph{physical propensity} of a given outcome in a given experiment, in lieu of an epistemic view where such statements only measure the degree to which we are uncertain about said outcome \citep{popper1959propensity}. To us the propensity interpretation seems adequate in the framework of coin-throwing, as it is meaningful to establish a connection between the coin's physical center of mass and the propensity of it landing `heads' in any one given throw: recalling that a coin throw is governed by classical (deterministic) mechanics, we could for instance let $\Omega$ denote the set of all possible initial conditions (angle, speed, spin, etc) and then make the requirement that the subset comprised of all initial conditions whose corresponding outcome is `heads' be a measurable set, with measure $p\in[0,1]$. Clearly such $p$ is a function of the coin's center of mass.}
that if a coin is thrown independently a large number of times, then the observed proportion of heads should not be far from the \emph{parameter of unbalancedness} $\theta\in[0,1]$ (this quantity being understood as representing the probability, or `chance', of observing heads in any one individual throw). In the Theory of Probability, the law of large numbers supports, generalizes and also provides a precise mathematical meaning to this intuition --- an intuition which can be traced back at least to Cardano's 16th-century \emph{Liber de ludo aleae} \citep{cardano2015book}. In his 1713 treatise \emph{Ars Conjectandi}, Jacob Bernoulli gave the first proof of the fact that (in modern notation) if $X$ is a Binomial random variable with parameters $n\in\N$ and $0\leq p\leq1$, then one has the inequality
\(\Prob(|n^{-1}X - p| > \varepsilon) \leq (1+c)^{-1},\)
provided $n$ is large enough, where $\varepsilon$ and $c$ are arbitrarily prescribed positive constants \citep{bernoulli2005law}. This is a typical \emph{weak law} statement --- although it was not until the time of Poisson that the name ``loi des grands nombres'' was coined \citep[p.7]{poisson1837recherches}. See \cite{seneta1992history,seneta2013history} for a compelling historical perspective on the law of large numbers, a history which culminated in `the' strong law for independent and identically distributed sequences, according to which the almost sure convergence of the sequence of sample means to the (common) expected value is equivalent to integrability. Also, still in the context of independent sequences, we highlight the importance of Kolmogorov's strong law for independent sequences whose partial sums have variances satisfying a summability condition.

Outside the realm of independence, things get trickier. As famously put by Michel Loève~\citep[p.6]{loeve1973paullevy}, ``martingales, Markov dependence and stationarity are the only three dependence concepts so far isolated which are sufficiently general and sufficiently amenable to investigation, yet with a great number of deep properties''. The contemporary probabilist would likely add uncorrelatedness, $m$-dependence, exchangeability and mixing properties to that list. In any case, the ways through which \emph{independence} may fail to hold are manifold, and thus one might infer that \emph{dependence} is too wide a concept, which means we should not expect to easily obtain a characterization of the law of large numbers for dependent sequences. Indeed, there are many scenarios where one can give \emph{sufficient} conditions under which a law of large numbers holds for such sequences --- to cite just a few examples: the weak law for pairwise uncorrelated sequences of random variables; the strong law for mixing sequences \citep{kuczmaszewska2011strong,kontorovich2014strong}; the strong law for exchangeable sequences \citep{taylor1987laws}; some very interesting results concerning decay of correlations (see, for example, \citep{Hu}) --- but, to the best of our knowledge, {no} \emph{characterization} has been provided so far\footnote{It is well known that the problem can be translated --- although not \emph{ipsis litteris} --- to the language of Ergodic Theory, and there are many characterizations of ergodicity of a dynamical system. The law of large numbers for stationary sequences is indeed implied by the Ergodic Theorem, but the converse implication does not hold in general.}. In this paper, we provide one such characterization for sequences of identically distributed Bernoulli random variables, in terms of the concept of a \emph{product disintegration}. Our main result shows that, to a certain degree, independence is an inextricable aspect of the law of large numbers.

Our conceptualization derives from --- and generalizes --- the notion of an \emph{exchangeable} sequence of random variables, to which we shall recall the precise definition shortly. First, let us get back to heuristics. The intuition underlying the coin-throwing situation depicted above remains essentially the same if we assume that, before fabricating the coin, the parameter of unbalancedness will be chosen at random in the interval $[0,1]$. In this case, conditionally on the value of the randomly chosen $\vartheta$ (let us say that the realized value is $\theta$), the long run proportion of heads definitely ought to approach $\theta$. The natural follow-up is to consider the not so evident scenario in which we choose at random (possibly distinct) parameters of unbalancedness $\vartheta_0,\dots,\vartheta_n, \dots$ and then, given a realization of these random variables (say, $\theta_0,\dots,\theta_n,\dots$), we fabricate distinct coins accordingly, that is, each corresponding to one of the sampled parameters of unbalancedness, and then sequentially throw them, independently from one another. Our main result implies that, if the sequence $(\vartheta_n)$ is stationary and satisfies a law of large numbers, then the long run proportion of heads in the latter scenario will approach $\E\vartheta_0$. Moreover, we show that the converse is also true: if a stationary sequence of coin throws has the property that the proportion of heads in the first $n$ throws approaches, with certainty, the parameter of unbalancedness, then the coin throws are conditionally independent, where the conditioning is on a sequence of \emph{random} parameters of unbalancedness satisfying themselves a law of large numbers.

As a byproduct stemming from our effort to provide a rigorous proof to Theorem~\ref{thm:LLN-characterization}, we developed the framework of \emph{product disintegrations}, which provides a model for sequences of random variables that are conditionally independent --- but not necessarily identically distributed --- thus being a generalization of exchangeability. In this context, we highlight the importance of Theorem~\ref{thm:prop1}, which constitutes the fundamental step in proving Theorem~\ref{thm:LLN-characterization} and also yields several examples that illustrate applications of both mathematical and statistical interest.

The paper is organized as follows. In the next section we state our main result, Theorem~\ref{thm:LLN-characterization}, and provide some heuristics connecting our conceptualization to the theory of \emph{exchangeable} sequences of random variables and to de Finetti's Theorem. In section~\ref{sec:general-theory}, we develop the theory in a slightly more general framework, introducing the concept of a \emph{product disintegration} as a generalization of exchangeability. We then state and prove our auxiliary results, of which Theorem~\ref{thm:LLN-characterization} is an immediate corollary. Section~\ref{sec:examples} provides a few examples.

\section{Main result and its relation to exchangeability}
We now state our main result. The proof is postponed to section \ref{sec:general-theory}.

\begin{theorem}\label{thm:LLN-characterization}
Let $\bs{X}\coloneqq\left(X_0,X_1,\dots\right)$ be a sequence of Bernoulli$(p)$ random variables, where $0\leq p\leq1$. Then one has
\begin{equation}\label{eq:LLN}
\lim_{n\to\infty}\frac{1}{n}\sum_{i=0}^{n-1} X_i = p,\qquad\mbox{almost surely}
\end{equation}
if and only if there exists a sequence $\bs\vartheta = (\vartheta_0, \vartheta_1,\dots)$ of random variables taking values in the unit interval such that:
\begin{enumerate}
\item almost surely, for all $n\geq0$ and all $x_0, x_1,\dots,x_n \in \{0,1\}$ one has
\begin{equation}\label{eq:bernoulli-product-disintegration}
\Prob(X_0 = x_0, \dots, X_n = x_n\given\bs\vartheta) = \prod_{i=0}^{n} \vartheta_i^{x_i}(1-\vartheta_i)^{1-x_i},
\end{equation}
and
\item almost surely, it holds that
\begin{equation}\label{eq:LLN-xi}
\lim_{n\to\infty}\frac1n\sum_{i=0}^{n-1} \vartheta_i = p.
\end{equation}
\end{enumerate}
\end{theorem}

\begin{remark}
The above theorem says that a sequence of coin throws has the property that the proportion of heads in the first $n$ throws approaches, with certainty, the ``parameter of unbalancedness'' $p\in[0,1]$ \emph{if and only if} the coin throws are conditionally independent, where the conditioning is on a sequence of \emph{random} parameters of unbalancedness whose corresponding sequence of sample means converges to $p$. Thus, for sequences of identically distributed Bernoulli$(p)$ random variables, the strong law of large numbers holds precisely when the experiment can be described as the outcome of a two-step mechanism, in which the first step encapsulates dependence and convergence of the sample means, whereas in the second step the random variables are realized in an independent manner.
\end{remark}

The conditional independence expressed in equation~\eqref{eq:bernoulli-product-disintegration} is closely related to the notion of \emph{exchangeability}. Recall that a sequence $\bs{X} \coloneqq (X_0, X_1,\dots)$ of random variables is said to be \defin{exchangeable} iff for every $n\geq1$ and every permutation $\sigma$ of $\{0,\dots,n\}$ it holds that the random vectors $(X_0,\dots,X_n)$ and $(X_{\sigma(0)},\dots, X_{\sigma(n)})$ are equal in distribution. An important characterization of exchangeability, de Finetti's Theorem states that a necessary and sufficient condition for a sequence of random variables to be exchangeable is that it is {\it conditionally independent and identically distributed}. To be precise, in the context of a sequence $\bs{X} \coloneqq (X_0, X_1,\dots)$ of Bernoulli$(p)$ random variables, exchangeability is equivalent to existence of a random variable $\vartheta$ taking values in the unit interval such that, almost surely, for all $n\geq0$ and all $x_0,\dots,x_n \in \{0,1\}$ one has
\begin{equation}\label{eq:productmeasure}
\Prob\left(X_0 = x_0,\dots,X_n = x_n\given\vartheta\right) = \prod_{i=0}^{n} \vartheta^{x_i}(1-\vartheta)^{1-x_i}.
\end{equation}
Moreover, $\vartheta$ is almost surely unique and given by $\vartheta = \lim_{n\to\infty} n^{-1}\sum_{i=0}^{n-1} X_i$.
In fact, the above equivalence holds with greater generality --- see \citep[Theorem~11.10]{kallenberg2002foundations}.

In view of de Finetti's Theorem, one is tempted to ask what happens when the random product measure \eqref{eq:productmeasure} characterizing exchangeable sequences --- whose factors are all the same random probability measure --- is substituted by an arbitrary random product measure (whose factors are not necessarily the same). This led us to introduce the concept of a \emph{product disintegration}, which we develop below, and which ultimately provided us with the framework yielding Theorem~\ref{thm:LLN-characterization}.

\section{General theory and proof of Theorem \ref{thm:LLN-characterization}}\label{sec:general-theory}
We now proceed to developing a  slightly more general theory --- one that will lead us to Theorem~\ref{thm:prop1}, of which Theorem~\ref{thm:LLN-characterization} is a corollary. Let us begin by establishing some terminology and notation. \textbf{In all that follows, $S$ is a compact, metrizable space}. We let $M_1(S)$ denote the set of Borel probability measures on $S$. The former is itself a compact metrizable space when endowed with the topology of weak* convergence --- according to which a sequence $(\mu_n)$ of probability measures converges to a given $\mu\in M_1(S)$ if and only if $\int f(x)\,\mu_n(\dd x) \to \int f(x)\, \mu(\dd x)$, for each continuous function $f\colon S\to\R$. In particular $M_1(S)$ admits a Borel $\sigma$-field --- see Theorem~\ref{thm:bogachev-8-3-2}. If $(\Omega, \mathscr{F}, \Prob)$ is a probability space and $\xi\colon\Omega \rightarrow M_1(S)$ is a Borel measurable mapping, we call $\xi$ a \defin{random probability measure on $S$}, whose value (which is a fixed probability measure) at a point $\omega\in\Omega$ we shall denote by $\xi^\omega$ and $\xi(\omega,\cdot)$ interchangeably. $\mathrm{Meas}_b(S)$ denotes the space of measurable, bounded maps from $S$ to $\R$, and $C\left(S\right)$ denotes the subspace of $\mathrm{Meas}_b(S)$ comprised of continuous maps from $S$ to $\R$. Given $f\in \mathrm{Meas}_b(S)$ and $\mu\in M_1\left(S\right)$ we shall write $\int f(x)\,\mu(\dd x)$, $\mu\left(f\right)$ and $\hat{f}(\mu)$ interchangeably. If $\xi$ is a random probability measure on $S$, the \defin{baricenter of $\xi$} is defined as the unique element $\E\xi\in M_1\left(S\right)$ such that the equality 
\(
\int_\Omega \int_S f(x)\xi^\omega(\dd x)\Prob(\dd \omega) = \int_S f(x) \E\xi(\dd x)
\) 
holds for all $f\in C\left(S\right)$. The baricenter $\E\xi$ is also known as the \defin{Pettis integral of $\xi$ with respect to $\Prob$}, or as the \defin{$\Prob$-expectation of $\xi$}, and its existence is guaranteed by the Riesz-Markov Theorem~\ref{thm:riesz-markov}. As usual, we write $\Prob_Y$ for the \emph{distribution} of a random variable $Y$ with values in a measurable space $M$, that is, $\Prob_Y(B) = P(Y\in B)$, for any measurable subset $B\subseteq M$. In what follows $\N$ denotes the set of nonnegative integers.

\begin{definition}[Product Disintegration]\label{def:product-disintegration-simple-version}
Let $\bs{X} \coloneqq \left(X_0,X_1,\dots\right)$ be a sequence of random variables taking values in a compact metric space $S$. We say that a sequence $\bs{\xi}\coloneqq\left(\xi_0, \xi_1,\dots\right)$ of random probability measures on $S$ is a \defin{product disintegration of $\bs{X}$} iff, with probability one, the equality
\begin{equation}\label{eq:weaker-product-measure-disintegration}
\Prob\left[X_0\in A_0,\dots,X_n\in A_n\given\bs{\xi}\right] = \xi_0\left(A_0\right)\cdots\xi_n\left(A_n\right)
\end{equation}
holds  for each $n\in\N$ and each family $A_0,\dots,A_n$ of measurable subsets of $S$. If $\bs{\xi}$ is a stationary sequence, then we say that $\bs{\xi}$ is a \defin{stationary product disintegration}.
\end{definition}

The definition above says that, conditionally on $\bs\xi$, the sequence $\bs{X} \coloneqq \left(X_0,X_1,\dots\right)$ is independent --- or, to be more precise, that for almost all elementary outcome $\omega$ in the sample space, it holds that the conditional probability $\Prob(\bs X\in \cdot\given\bs\xi)_\omega$ is a product measure on $S^\N$. See the \emph{standard construction} below for more details, {where a justification for the terminology \emph{disintegration} is provided}. Also, notice that if $\bs\xi$ is stationary, then clearly $\bs X$ is stationary as well.

The following result is an important characterization of product disintegrations. It allows us to work with the seemingly weaker requirement that the identity~\eqref{eq:weaker-product-measure-disintegration} hold only on a set $\Omega[n; A_0,\dots, A_n]$ having $\Prob$-measure $1$, for each $n\in\N$ and each family $A_0,\dots,A_n$ of measurable subsets of $S$.

\begin{lemma}\label{thm:product-measure-disintegration}
Let $\bs{X} \coloneqq \left(X_0,X_1,\dots\right)$ be a sequence of random variables taking values in a compact metric space $S$, and let $\bs{\xi} = (\xi_0,\xi_1,\dots)$ be a sequence of random probability measures on $S$. Then $\bs\xi$ is a product disintegration of $\bs X$ if and only if for each $n$ and each $(n+1)$-tuple $A_0,\dots,A_n$ of measurable subsets of $S$, the equality \eqref{eq:weaker-product-measure-disintegration} holds almost surely.
\end{lemma}

\begin{proof}
The `only if' part of the statement is trivial. For the `if' part, let $\mathscr{S}^\N$ denote the product $\sigma$-field on $S^\N$. By Lemma~\ref{thm:product-topology}, $\mathscr{S}^\N$ coincides with the Borel $\sigma$-field corresponding to the product topology on $S^\N$, and therefore $S^\N$ is a Borel space. By Theorem~\ref{thm:regular-conditional-distribution}, there exists an event $\Omega^*\subseteq\Omega$ with $\Prob(\Omega^*)=1$ such that
\(
\bs{A}\mapsto \Prob(\bs{X}\in\bs{A}\given\bs{\xi})_\omega
\)
is a probability measure on $\mathscr{S}^\N$ for each $\omega\in\Omega^*$.

Now let $\mathfrak{C}\coloneqq\{\bs{A}_k\colon\,k\in\mathbb{N}\}$ be a countable collection of sets of the form $\bs{A}_k = B_0^k\times\cdots\times B_{n(k)}^k\times S\times\cdots$ which generates $\mathscr{S}^\N$ (see Corollary~\ref{thm:subbasis}). By assumption, for each $k$ there is an event $\Omega_k\subseteq \Omega$ with $\Prob(\Omega_k)=1$ such that
\(
\Prob(\bs{X}\in\bs{A}_k\given\bs{\xi})_\omega = \xi_0^\omega(B_0^k)\cdots\xi_{n(k)}^\omega(B_{n(k)}^k)
\)
holds for $\omega\in\Omega_k$. Thus, for $\omega\in \Omega'\coloneqq \left(\bigcap_{k=0}^\infty \bs{A}_k\right)\cap \Omega^*$, with $\Prob(\Omega')=1$, the probability measures $\Prob(\bs{X}\in\cdot\given\bs{\xi})_\omega$ and $\prod_{n=0}^\infty\xi_n^\omega$ agree on a $\pi$-system which generates $\mathscr{S}^\N$, and therefore they agree on $\mathscr{S}^\N$. This establishes the stated result.
\end{proof}

Now we prove that product disintegrations always exist:

\begin{lemma}\label{thm:product-disintegration-existence}
Any sequence $\bs{X}\coloneqq (X_0, X_1,\dots)$ of $S$-valued random variables admits a product disintegration.
\end{lemma}

\begin{proof}
For $n\in\N$ and $\omega\in\Omega$, let $\xi_n^\omega = \delta_{X_n(\omega)}$, where $\delta_x$ is the Dirac measure at $x\in S$. Now fix $n\in\N$ and let $A_0,\dots,A_n$ be measurable subsets of $S$. We first prove that the map
\begin{equation}\label{eq:canonical-disitegration-is-measurable}
\omega\mapsto \xi_0^\omega\left(A_0\right)\cdots\xi_n^\omega\left(A_n\right)\equiv \I_{\left[X_0\in A_0,\dots,X_n\in A_n\right]}\left(\omega\right)
\end{equation}
is $\sigma\left(\bs{\xi}\right)$-measurable and integrable: by Theorem~\ref{thm:KAL-A2-3}, the maps $f_{A_i}\colon M_1(S)\to \R$ defined by $f_{A_i}(\mu)\coloneqq\mu(A_i)$, are measurable and thus, by the Doob-Dynkin Lemma~\ref{thm:doob-dynkin}, the map $\omega \mapsto \xi_i^\omega(A_i) = f_{A_i}\circ\xi_i(\omega)$ is measurable with respect to $\sigma(\xi_i)\subseteq\sigma(\bs{\xi})$. Thus \eqref{eq:canonical-disitegration-is-measurable} defines a $\sigma(\bs{\xi})$-measurable map, as stated. Moreover, for $B\in\sigma\left(\bs{\xi}\right)$ we have
\begin{equation*}
\E\left\{\xi_0\left(A_0\right)\cdots\xi_n\left(A_n\right)\I_B\right\} = \E\left\{\I_{\left[X_0\in A_0,\dots,X_n\in A_n, B\right]}\right\} = \Prob\left\{X_0\in A_0,\dots,X_n\in A_n, B\right\},
\end{equation*}
and therefore $\xi_0\left(A_0\right)\cdots\xi_n\left(A_n\right)$ is a version of $\Prob\left[X_0\in A_0,\dots,X_n\in A_n\given\bs{\xi}\right]$. Now it is only a matter of applying Lemma~\ref{thm:product-measure-disintegration}.
\end{proof}

We shall call the sequence $\bs\delta = \left(\delta_{X_0},\delta_{X_1},\dots\right)$ appearing in the above lemma the \defin{canonical product disintegration of $\bs{X}$}. Notice, in particular, that product disintegrations are not unique (see Example~\ref{example:non-unique}). Also, it is clear that stationarity of $\bs X$ entails stationarity of $\bs\delta$.

\textbf{We now argue that}, without loss of generality, one can take the underlying probability space $\Omega$ to be the compact metric space $S^\N\otimes M_1(S)^\N$, endowed with its Borel $\sigma$-field $\mathscr{F}$, and equipped with the probability measure defined, for Borel subsets $\bs{A}\subseteq S^\N$ and $\bs{B}\subseteq M_1(S)^\N$, by
\begin{equation}\label{eq:kernel-equation}
\Prob(\bs{A}\times\bs{B}) = \int_{\bs{B}} \rho(\bs{\lambda},\bs{A})\,Q(\dd \bs{\lambda})
\end{equation}
where $Q$ is a probability measure defined on $M_1(S)^\N$ (that is, $Q\in M_1(M_1(S)^\N)$) and
\begin{equation*}
\rho(\bs{\lambda},\bs{A}) \coloneqq \left(\prod\nolimits_{i\in\N}\lambda_i\right)(\bs{A}),\qquad \bs\lambda\in M_1(S)^\N,\quad \bs A\subseteq S^\N\,\text{measurable}.
\end{equation*}
In this construction, the random variables $\bs{X}$ and $\bs{\xi}$ can be defined as projections by putting, for $\omega = (\bs{x},\bs{\lambda})\in\Omega$, $\bs{X}(\omega) \coloneqq \bs{x}$ and $\bs{\xi}(\omega) \coloneqq \bs{\lambda}$, where $\bs{x} = (x_0, x_1, \dots)$ and $\bs{\lambda} = (\lambda_0, \lambda_1,\dots)$. The next lemma ensures that, in the probability space $(\Omega, \mathscr{F},\Prob)$, indeed $\bs{\xi}$ is a product disintegration of $\bs{X}$, with $\Prob_{\bs{\xi}}=Q$. For convenience, we shall call this the \defin{standard construction}.

\begin{lemma}\label{thm:rho-is-measurable}
$\rho$ is a probability kernel from $M_1(S)^\N$ to $S^\N$.
\end{lemma}
\begin{proof}
It is sufficient to prove that the map
\(\bs\lambda\mapsto\rho(\bs\lambda,\cdot)\equiv \prod_{i\in\N}\lambda_i\)
from $M_1(S)^\N$ to $M_1(S^\N)$ is measurable. Let $(\bs{\lambda^n})$ be a sequence in $M_1(S)^{\mathbb{N}}$, i.e., for each $n$, $\bs{\lambda^n} = (\lambda^n_0, \lambda^n_1, \dots)$ with $\lambda^n_i \in M_1(S)$, for each $i$, such that $\lim_{n \to \infty} \bs{\lambda^n} = \bs{\lambda} = (\lambda_0, \lambda_1, \dots) \in M_1(S)^{\mathbb{N}}$; that is, $\lim_{n \rightarrow +\infty} \lambda^n_i = \lambda_i$, for all $i$. Also, let $\bs A = A_0 \times A_1 \times \dots \times A_L \times S \times S \times \dots$ be an open set in $S^{\N}$. Since $\lim_{n \rightarrow +\infty} \lambda^n_i = \lambda_i$, we know, by the Portmanteau Theorem, that $\liminf_{n\rightarrow+\infty} \lambda^n_i (A_i) \geq \lambda_i(A_i)$. Now,
\(
\rho\left(\bs{\lambda}^n, \bs A\right) = \left(\prod_{j \in\N}\lambda^n_j\right)(\bs A) = \prod_{j=0}^{L} \lambda^n_j(A_j).
\)
This implies
\begin{equation*}
\liminf_{n \to +\infty} \rho\left(\bs{\lambda}^n, \bs A\right) = \liminf_{n \to +\infty} \prod_{j=0}^{L} \lambda^n_j(A_j) = \prod_{j =0}^{L} \liminf_{n \to +\infty} \lambda^n_j(A_j) \geq \prod_{j=0}^{L} \lambda_j(A_j) = \rho\left(\bs{\lambda}, \bs A\right),
\end{equation*}
which proves that $\bs\lambda\mapsto\rho(\bs\lambda,\cdot)$ is continuous and, \emph{a fortiori}, measurable.
\end{proof}

Interestingly, the standard construction evinces the fact that the joint law of a sequence of random variables with values in $S$ can \emph{always} be written as the baricenter of a random product measure on $S^\N$. Indeed, as product disintegrations always exist (Lemma~\ref{thm:product-disintegration-existence}), if we let $\bs{X} = (X_0, X_1,\dots)$ be such a sequence (and seeing $\bs X$ as a $S^\N$-valued random variable) with product disintegration $\bs\xi = (\xi_0,\xi_1,\dots)$, then, writing $\rho(\bs\lambda)\equiv\rho(\bs\lambda,\cdot)$,  we have
\begin{align*}
\Prob_{\bs{X}} = \E\left( \prod\nolimits_{n=0}^\infty \xi_n \right) = \int \rho\circ\bs{\xi}(\omega)\,\Prob(\dd\omega) = \int \rho(\bs{\lambda})\,\Prob_{\bs{\xi}}(\dd\bs{\lambda})
\end{align*}
and, of course, 
\(\Prob_{\bs\xi}\lbrace \bs\lambda\colon\,\text{$\rho(\bs\lambda)$ is a product measure}\rbrace = 1.\) Moreover, the standard construction justifies the adoption of the terminology \emph{product disintegration}; indeed, in this setting the family of probability measures $(\eta^\omega\colon \omega\in\Omega)$ defined on $(\Omega,\mathscr F)$ via
\[
\eta^\omega(\bs A\times\bs B) \coloneqq \rho\big(\bs\xi(\omega),\bs A\big)\,\I_{[\bs\xi\in\bs B]}(\omega)\equiv \Prob(\bs A\times\bs B\given \bs\xi)_\omega,
\]
for measurable sets $\bs A\subseteq S^\N$ and $\bs B\subseteq M_1(S)^\N$, provides a disintegration of $\Prob$ with respect to $\sigma(\bs\xi)$. See the definition 10.6.1 in \cite{bogachev2007measure} and also the proof of Theorem~\ref{thm:prop1} for more details.

Theorem~\ref{thm:LLN-characterization} is a direct consequence of Theorem~\ref{thm:prop1} below. The `if' part of this proposition is inspired by a similar result that has appeared --- albeit in a different framework --- in \citep[Theorem~1]{horta2018conjugate}.\footnote{The reasoning used by the authors in their proof is essentially the same as the one we apply here, although their statement corresponds to a weak law whereas ours is a strong law. We also made an effort to provide the measure theoretic details in the argument.} Its proof relies on the following \emph{disintegration theorem}.

\begin{theorem}\label{thm:bogachev-rcp}
Let $\Omega$ and $\Lambda$ be compact metric spaces, let $\Prob$ be a Borel probability measure on $\Omega$, and let $\bs\xi\colon\Omega\to \Lambda$ be a Borel mapping. Then there exists a collection $(\eta^{\bs\lambda}\colon\, \bs\lambda\in \Lambda)$ of Borel probability measures on $\Omega$ such that
\begin{enumerate}
\item the functions $\bs\lambda\mapsto\eta^{\bs\lambda}(E)$ are Borel measurable, for each measurable subset $E\subseteq \Omega$.
\item one has $\eta^{\bs\lambda} \{\omega:\,\bs\xi\left(\omega\right)\neq \bs\lambda\} = 0$, for every $\bs\lambda\in \mathrm{range}(\bs\xi)$.
\item for all measurable subsets $E\subseteq\Omega$ and $\bs L\subseteq \Lambda$ one has
\(
\Prob(E\cap\bs\xi^{-1}(\bs L)) = \int_{\bs L} \eta^{\bs\lambda}(E)\,\Prob_\xi(\dd \bs\lambda).
\)
\end{enumerate}
\end{theorem}
\begin{proof}
This is a direct consequence of Proposition~10.4.12 in \cite{bogachev2007measure}.
\end{proof}

\begin{remark}
In the context of the above theorem, it is commonplace to write $\eta^{\bs\lambda}(E)=:\Prob(E\given\bs\xi=\bs\lambda)$, in which case the above theorem yields the \defin{substitution principle}, $\Prob\{\omega\colon\,g(\omega,\bs\xi(\omega))=g(\omega,\bs\lambda)\given\bs\xi=\bs\lambda\} = 1$ for all $\bs\lambda\in\mathrm{range}(\bs\xi)$ and all measurable functions $g$ defined on $\Omega\times \Lambda$. The probability kernel appearing in the above theorem is essentially unique: indeed, if $(\eta_1^{\bs\lambda}\colon\bs\lambda\in\Lambda)$ is another such kernel, then it is easy to see that $\eta_1^{\bs\lambda} = \eta^{\bs\lambda}$ for $\bs\lambda$ on a set of total $\Prob_{\bs\xi}$-measure.
\end{remark}

\begin{theorem}\label{thm:prop1}
Let $\bs{X} = \left(X_0, X_1,\dots\right)$ be a sequence of $S$-valued random variables. Assume $\bs{\xi} = (\xi_0, \xi_1,\dots)$ is a product disintegration of $\bs{X}$, and let $f\in C(S)$. Then it holds that
\begin{equation}\label{eq:prop1}
\lim_{n\to\infty} \frac1n\sum_{i=0}^{n-1}\big(f\circ X_i - \xi_i(f)\big) = 0
\end{equation}
almost surely. In particular, the limit 
\[X_\infty(f)\coloneqq \lim_{n\to\infty} n^{-1}\sum_{i=0}^{n-1}f\circ X_i\] 
exists almost surely if and only if the limit 
\[\xi_\infty(f)\coloneqq \lim_{n\to\infty}n^{-1}\sum_{i=0}^{n-1}\xi_i(f)\] exists almost surely, in which case one has $X_\infty(f) = \xi_\infty(f)$ almost surely.
\end{theorem}

\begin{remark}
Notice that, in the theorem above, no additional assumptions are imposed on the product disintegration $\bs\xi$. In particular, Theorem~\ref{thm:prop1} holds when $\bs\xi$ is the canonical product disintegration of $\bs X$. This is crucial for the `only if' part of Theorem~\ref{thm:LLN-characterization}.
\end{remark}

\begin{remark}
For simplicity, we just ask  $f\in C(S)$ in the statement of Theorem \ref{thm:prop1}, and in fact this is all we need in the following results and also in the examples of section \ref{sec:examples}, but we remark that the result also holds if $f$ is only assumed to be measurable and bounded.
\end{remark}

The corollary below is an immediate consequence of Theorem~\ref{thm:prop1}, by taking $S$ as a compact subset of the real line and $f$ as the identity map (in which case $\xi_i(f) = \int_S x\, {\xi}_i(\dd x) = \E(X_i \given \bs{\xi})$), and shows how the product disintegration can be used to assure the validity of the strong law of large numbers for a sequence of {uniformly bounded} random variables.

\begin{corollary}
Suppose $S$ is a compact subset of the real line, and assume $\bs{\xi}\coloneqq (\xi_0,\xi_1,\dots)$ is a product disintegration of $\bs{X} \coloneqq (X_0, X_1,\dots)$, where the $X_i$ are random variables with values in $S$. Then the limit $X_\infty\coloneqq \lim_{n\to\infty}{n}^{-1}\sum_{i=0}^{n-1}  X_i$ exists almost surely if and only if the limit $\xi_\infty \coloneqq \lim_{n\to\infty} n^{-1}\sum_{i=0}^{n-1} \E(X_i \given \bs{\xi})$ exists almost surely, in which case $X_\infty = \xi_\infty$ a.s. If moreover $\xi_\infty^\omega$ does not depend on $\omega$ (almost surely), then the strong law of large numbers holds for $\bs{X}$.
\end{corollary}

\begin{proof}[Proof of Theorem~\ref{thm:prop1}]
Write
\(Z_i \coloneqq f\circ X_i - \xi_i(f).\)
We have
\begin{equation}\label{eq:WLLN-2}
\Prob\bigg(\lim_{n \to \infty}\abs{n^{-1}\sum\nolimits_{i=0}^{n-1} Z_i} = 0 \bigg) = \E\left\{\Prob\left( \lim_{n \to \infty}\abs{n^{-1}\sum\nolimits_{i=0}^{n-1} Z_i} = 0 \,\,\bigg\vert\,\,\bs{\xi}\right)\right\}.
\end{equation}
The idea now is that $\left(Z_n\given\bs{\xi}:\,n\in\N\right)$ is an independent sequence, with $\E\left[Z_n\given\bs{\xi}\right] = 0$ and $\sup_n\Var\left(Z_n\given\bs{\xi}\right)\leq4\Vert f\Vert_\infty^2<\infty$, and therefore Kolmogorov's strong law (Theorem~\ref{thm:kolmogorov-slln}) ensures that, with probability one, the conditional probability inside the expectation in \eqref{eq:WLLN-2} is equal to 1.

To make this argument precise, take $\Omega$, $\Prob$, $\bs X$ and $\bs\xi$ as in the standard construction discussed above, and let $\big(\eta^{\bs\lambda}\colon\,\bs\lambda\in M_1(S)^{\mathbb{N}}\big)$ be given as in Theorem~\ref{thm:bogachev-rcp}, with $\Lambda = M_1(S)^\N$. In this setting it is easy to see that, for $E\subseteq\Omega$ of the form $E = \bs A\times\bs B$, with $\bs A\subseteq S^\N$ and $\bs B\subseteq M_1(S)
^\N$, we have $\eta^{\bs\lambda}(E) = \rho(\bs\lambda,\bs A)\I_{\bs B}(\bs\lambda)$. Indeed, here we have $(\bs A\times\bs B)\cap\bs\xi^{-1}(\bs L) = \bs A\times(\bs B\cap \bs L)$ and then, by \eqref{eq:kernel-equation}, 
\begin{equation*}
\Prob\big((\bs A\times\bs B)\cap\bs\xi^{-1}(\bs L)\big) = \int_{\bs B\cap \bs L}\rho(\bs\lambda,\bs A)\,\Prob_{\bs\xi}(\dd \bs\lambda) = \int_{\bs L}\rho(\bs\lambda,\bs A)\,\I_{\bs B}(\bs \lambda)\,\Prob_{\bs\xi}(\dd \bs\lambda).
\end{equation*}
In particular,
\begin{equation}\label{eq:eta-in-terms-of-rho-2}
\eta^{\bs\lambda}(\bs A\times M_1(S)) = \rho(\bs\lambda, \bs A).
\end{equation}
Now let $E = \big\{\omega\colon\,\lim\abs{n^{-1}\sum\nolimits_{i=0}^{n-1} Z_i(\omega)} = 0\big\}$. By Theorem~\ref{thm:bogachev-rcp}, we have
\(
\Prob(E) = \int \eta^{\bs\lambda}(E)\,\Prob_{\bs\xi}(\dd\bs\lambda)
\)
and
\begin{equation}\label{eq:eta-lambda-equality}
\eta^{\bs\lambda}(E) = \eta^{\bs\lambda}\big\{\omega\colon\,\lim \vert n^{-1}\sum\nolimits_{i=0}^{n-1} f\circ X_i(\omega) - \lambda_i(f)\vert = 0\big\},
\end{equation}
Thus, writing $\bs A_{\bs\lambda} = \big\{\bs x\in S^\N\colon\,\lim \vert  n^{-1}\sum\nolimits_{i=0}^{n-1} f(x_i) - \lambda_i(f)\vert =0\big\}$, we see that the following equality of events holds
\[
\bs A_{\bs\lambda} \times M_1(S)^\N = \big\{\omega\colon\, \lim \vert n^{-1}\sum\nolimits_{i=0}^{n-1} f\circ X_i(\omega) - \lambda_i(f)\vert = 0\big\}. 
\]
Therefore, by \eqref{eq:eta-in-terms-of-rho-2} and \eqref{eq:eta-lambda-equality}, we obtain
\(\eta^{\bs\lambda}(E) = \rho(\bs\lambda, \bs A_{\bs\lambda}) = 1,\)
where the rightmost equality follows from Kolmogorov's strong law, as $\rho(\bs\lambda,\cdot)$ is the law of a sequence of independent, zero mean random variables with uniformly bounded variances. This establishes \eqref{eq:prop1}. The second part of the statement now follows trivially.
\end{proof}

\begin{proof}[Proof of Theorem~\ref{thm:LLN-characterization}]
Recall that $S=\{0,1\}$. The idea is that in this setting $M_1(S)$ is isomorphic to the unit interval. First, notice that given any two probability measures $\lambda,\mu \in M_1(S)$, we have that $\lambda\neq\mu$ iff $\lambda\{1\}\neq\mu\{1\}$. Thus, the mapping $\lambda\mapsto f_1(\lambda):=\lambda\{1\}$ is one-to-one from $M_1(S)$ onto $[0,1]$. As Theorem~\ref{thm:KAL-A2-3} tells us that this mapping is measurable, we can apply Kuratowski's \emph{range and inverse} Theorem~\ref{thm:kuratowski} to conclude that its inverse is also measurable.

For the `if' part of the theorem, let $\xi_n^\omega$ be the unique probability measure in $M_1(S)$ for which $\xi_n^\omega\{1\} = \vartheta_n(\omega)$, $n\in\N$, $\omega\in\Omega$. The reasoning in the preceding paragraph then tells us that $\sigma(\vartheta_n) = \sigma(\xi_n)$ for all $n$ and consequently $\sigma(\bs\vartheta) = \sigma(\bs\xi)$. Therefore, we have that
\(
\Prob(\bs X\in \cdot \given\bs\xi)=\Prob(\bs X\in \cdot\given\bs\vartheta),
\)
which tells us that $\bs\xi$ is a product disintegration of $\bs X$ since the righthandside in this equality is a product measure on $S^\N$ (with probability 1), by assumption. As we have, again by assumption, that
$\lim_{n\to\infty}n^{-1}\sum_{i=0}^{n-1}\xi_i(f) = p$ almost surely, with $f = \I_{\lbrace1\rbrace}$ (which is a continuous function on $S$), Theorem~\ref{thm:prop1} then tells us that \eqref{eq:LLN} holds.

For the `only if' part, let now $\bs\xi=(\xi_0,\xi_1,\dots)$ denote the canonical product disintegration of $\bs X$, and write $\vartheta_n := \xi_n\{1\}$ for all $n$. It is clear (again using the fact that $\sigma(\bs\xi)=\sigma(\bs\vartheta)$) that \eqref{eq:bernoulli-product-disintegration} holds.
Also, we have by assumption that $\lim_{n\to\infty} n^{-1} \sum_{i=0}^{n-1} f\circ X_i = p$, with $f = \I_{\lbrace1\rbrace}$ (as $X_i = \I_{[X_i=1]}$), and since $\xi_i(f) \equiv\xi_i\{1\} = \vartheta_i$, Theorem~\ref{thm:prop1} tells us that the limit in \eqref{eq:LLN-xi} holds. This completes the proof.
\end{proof}

\section{Examples}\label{sec:examples}

\subsection{Product disintegrations \emph{per se}}
\begin{example}[Product disintegrations are not (necessarily) unique]\label{example:non-unique}
Let $\bs\vartheta=(\vartheta_n:\,n\in\mathbb{N})$ be a sequence of independent and identically distributed random variables, uniformly distributed in the unit interval $[0,1]$, and let, for $n\in\N$, $\xi_n$ be the random probability measure on $S\coloneqq\lbrace0,1\rbrace$ defined via $\xi_n^\omega(1) \coloneqq \vartheta_n(\omega)$, where for simplicity we write $\xi_n(x)$, $x\in S$, instead of $\xi_n(\{x\})$. Assume further that $\bs\xi$ is a product disintegration of a given sequence $\bs X$ of Bernoulli random variables --- if necessary, proceed with the standard construction. As argued in the proof of Theorem~\ref{thm:LLN-characterization}, we have $\sigma(\bs\xi) = \sigma(\bs\vartheta)$ and, in particular, it holds that conditionally on $\bs{\xi}$ each $X_n$ is a Bernoulli random variable with parameter $\vartheta_n$. That is, for each $n\in\N$ we have
\(
\Prob(X_n = 1\given \bs{\xi}) = \vartheta_n. 
\)
Now define $\hat{\xi}_n\colon\Omega\rightarrow M_1(S)$ by $\hat{\xi}_n^\omega(1) \coloneqq \delta_{X_n(\omega)}(1) = \I_{[X_n=1]}(\omega)$, so that $\hat{\bs{\xi}}\coloneqq (\hat{\xi}_0,\hat{\xi}_1,\dots)$ is the canonical product disintegration of $\bs{X}$. Clearly $\bs{\xi}$ and $\hat{\bs{\xi}}$ are different since $\hat{\xi}_n^\omega$ is equal either to $\delta_{\lbrace0\rbrace}$ or $\delta_{\lbrace1\rbrace},$ whereas this is not true of $\xi_n$. Indeed, for $\theta\in[0,1)$, we have \(\Prob(\xi_n(1)\leq\theta) = \Prob(\vartheta_n\leq\theta) = \theta\),
whereas
\(
\Prob(\hat{\xi}_n(1)\leq \theta) = \Prob(\I_{[X_n=1]}\leq\theta)
= \Prob(\I_{[X_n=1]} = 0)
= \Prob(X_n = 0).
\)
\end{example}

\begin{example}[Random Walk as a two-stage experiment with random jump probabilities]
In the same setting as Example~\ref{example:non-unique}, let $Z_n \coloneqq 2X_n-1$, $n\in\N$. Clearly $\bs Z = (Z_0, Z_1,\dots)$ is an independent and identically distributed sequence of standard Rademacher random variables, i.e., for each $n \in \mathbb{N}$ it holds that $\Prob(Z_n = +1) = \Prob(Z_n = -1) = \nicefrac{1}{2}$. Indeed,
for any $x_0,x_1,\dots,x_n\in \{0,1\}$, we have
\(
\Prob(Z_0 = 2x_0-1, \dots, Z_n = 2x_n-1) =
\Prob\left(X_0 =  x_0  , \dots, X_n =  x_n \right)
= \E \prod_{j=0}^n \xi_j(x_j)
= \prod_{j=0}^n \E\xi_j(x_j),
\)
where the last equality follows from the assumption that the $\vartheta_n$'s are independent. Moreover, $\E \xi_j(x_j) = 1/2$ since the left-hand side in this equality is either $\E\vartheta_j$ or $1 - \E\vartheta_j$. Now let $S_0\coloneqq 0$ and $S_n = Z_0 + \cdots + Z_{n-1}$ for $n\geq 1$. By the above derivation, $(S_n\colon\,n\geq 0)$ is the symmetric random walk on $\Z$. Therefore, although --- unconditionally --- at each step the process $(S_n)$ jumps up or down with equal probabilities, we have that conditionally on $\bs\xi$ it evolves according to the following rule: at step $n$, sample a Uniform$[0,1]$ random variable $\vartheta_n$ independent of anything that has happened before (and of anything that will happen in the future), and go up with probability $\vartheta_n$, or down with probability $1 - \vartheta_n$.
\end{example}

\begin{example}
Let $\bs{X} = (X_0, X_1, \dots)$ be an exchangeable sequence of Bernoulli$(p$) random variables. In particular, $\bs X$ satisfies equation \eqref{eq:productmeasure} for some random variable $\vartheta$ taking values in the unit interval. Then, defining the random measures $\xi_n$ via $\xi_n(\{1\})\coloneqq \vartheta$ for all $n$, it is clear that $(\xi_0,\xi_1,\dots) =:\bs\xi$ is a stationary product disintegration of $\bs{X}$ --- again using the fact that $\sigma(\bs\xi) = \sigma(\bs\vartheta)$. In particular, in this scenario, an unconditional strong law of large numbers \emph{does not hold} for $\bs{X}$, unless when $\vartheta$ is a constant. See also Theorem~2.2 in \cite{taylor1987laws}, which provides a characterization of the strong law for the class of integrable, exchangeable sequences. This example illustrates that the existence of a product disintegration is not sufficient for the law of large numbers to hold (indeed, by Proposition~\ref{thm:product-disintegration-existence}, any sequence of random variables admits a product disintegration).
\end{example}

\begin{example}[Concentration inequalities]\label{example:concentration}
One important consequence of the notion of a product disintegration is that it allows us to easily translate certain concentration inequalities (such as the Chernoff bound, Hoeffding's inequality, Bernstein's inequality, etc) from the independent case to a more general setting. Recall that the classical Hoeffding inequality says that,
if $\bs X = (X_0, X_1,\dots)$ is a sequence of independent random variables with values in $\left[0, 1\right]$, then one has the bound
\(
\Prob\left(S_n \geq t\right)\leq\exp\left(-{2t^2}/{n}\right)
\)
for all $t>\E S_n$, where $S_n\coloneqq\sum_{i=0}^{n-1} X_i$.

\begin{theorem}[Hoeffding-type inequality]\label{thm:hoeffding-new}
Let $\bs X = \left(X_0, X_1,\dots\right)$ be a sequence of random variables with values in the unit interval $S := \left[0,1\right]$, and let $\bs\xi = (\xi_0,\xi_1,\dots)$ be a product disintegration of $\bs X$. Then, for any $t>0$, it holds that
\(
\Prob\left(S_n \geq t\given\E(S_n\given\bs\xi) < t\right) \leq\exp\left(-{2t^2}/{n}\right),
\)
where $S_n\coloneqq\sum_{i=0}^{n-1} X_i$.
\end{theorem}
\begin{proof}
From the classical Hoeffding inequality applied to the probability measures $\Prob(\cdot\given\bs\xi)_\omega$, we have
\(
\Prob\left( S_n \geq t\given\bs{\xi}\right)\I_{\lbrace\E(S_n \given\bs\xi) < t\rbrace}
\leq \exp\left(-{2t^2}/{n}\right)\I_{\lbrace\E(S_n \given\bs\xi) < t\rbrace}.
\)
Taking the expectation on both sides of the above inequality, and dividing by $\Prob(\E(S_n\,\given\,\bs\xi)< t)$, yields the stated result.
\end{proof}

Notice that if $\bs\xi$ is the canonical product disintegration of $\bs X$, then the above theorem is not very useful: indeed in this case we have $\E(S_n\given\bs\xi) = S_n$, so the left-hand side in the inequality is zero. The above theorem also tells us that, for $t>0$,
\begin{align*}
\Prob\big(S_n \geq t\big) &= \Prob\big(S_n \geq t\,\big\vert\,\E(S_n\given\bs\xi) < t\big)\times\Prob\big(\E(S_n\given\bs\xi) < t\big)\\
&\qquad\qquad +\quad\Prob\big(S_n \geq t\,\big\vert\,\E(S_n\given\bs\xi) \geq t\big)\times\Prob\big(\E(S_n\given\bs\xi) \geq t\big)\\
&\qquad\qquad\qquad\leq\quad \exp\left(-\frac{2t^2}{n}\right) \,+\, \Prob\big(\E(S_n\given\bs\xi) \geq t\big)
\end{align*}
so the rate at which $\Prob\big(S_n \geq t\big)\to0$ as $t\to\infty$ is governed by the rate at which $\Prob\left(\E(S_n\,\big\vert\,\bs\xi) \geq t\right)\to\infty$ as $t\to\infty$. To illustrate, let us consider two extreme scenarios, one in which $\xi_n = \xi_0$ for all $n$ (so that $\bs X$ is exchangeable) and one in which the $\xi_n$'s are all mutually independent: in the first case, we have that $\E\big(S_n\given\bs\xi\big) = n \int_0^1 x\,\xi_0(\dd x)$, and thus the rate at which $\Prob\big(S_n \geq t\big)\to0$ as $t\to\infty$ depends only on the distribution of the random variable $\int_0
^1 x\,\xi_0(\dd x)$. On the other hand, if the $\xi_n$'s are independent, then we have that $\E\big(S_n\given\bs\xi\big) = \sum_{i=0}^{n-1}\int_0^1 x\,\xi_n(\dd x)$, and in this case the summands are independent random variables with values in the unit interval. Therefore, we can apply the classical Hoeffding inequality to these random variables to obtain the upper bound $\Prob(S_n\geq t)\leq 2\exp(-2t
^2/n)$ for $t>\E S_n$ (in fact, we already know that the upper bound $\exp(-2t^2/n)$ holds, since independence of the $\xi_n$'s entails independence of the $X_n$'s).
\end{example}

\begin{example}
Let $S\coloneqq [a,b]^d$ where $d$ is a positive integer and $a<b\in\R$. Given a sequence $\bs{X} = (X_1, X_2, \dots)$ of $S$-valued random variables, we shall write $X_n = (X_n^{1},\dots,X_n^{d})$. Suppose $\bs\xi = (\xi_0,\xi_1,\dots)$ is a product disintegration of $\bs X$. Equation \eqref{eq:weaker-product-measure-disintegration} then yields, for all measurable sets $A_i^j\subseteq [a,b]$, with $i\in\lbrace0,\dots,n\rbrace$ and $j\in\lbrace1,\dots,d\rbrace$, the equality
\begin{equation*}
\Prob(X_0^1\in A_0^1,\dots X_0^d\in A_0^d,\dots, X_n^1\in A_n^1,\dots,X_n^d\in A_n^d\given\bs{\xi})
=\xi_0(A_0^1\times\cdots\times A_0^d)\cdots\xi_n(A_n^1\times\cdots\times A_n^d).
\end{equation*}
An identity as above appears naturally in statistical applications, for instance when one observes samples of size $d$, $(X_n^1,\dots,X_n^d)$, $n=0,1,\dots$, from distinct ``populations'' $\xi_0, \xi_1,\dots$ --- we refer the reader to \cite{petersen2016functional} and references therein for details.
\end{example}

\subsection{Convergence}
\begin{example}[Regime switching models]\label{exmp:regime-switch}
Let $S = \{-1,1\}$ and put $M'\coloneqq \{\mu,\lambda\}\subseteq M_1(S)$ with $\mu(1)>\lambda(1)$. The measures $\mu$ and $\lambda$ are to be interpreted as $2$ distinct ``{regimes}'' (for example, expansion and contraction, in which case one would likely assume $\mu(1)>1/2>\lambda(1)$). Let $(Q_{ij}\colon\,i,j\in M')$ be a row stochastic matrix with stationary distribution $\pi = (\pi_\mu, \pi_\lambda)$. Let $\bs{\xi}\coloneqq (\xi_0,\xi_1,\dots)$ be a Markov chain with state space $M'$, initial distribution $\pi$ and transition probabilities $(Q_{ij})$. Notice that $\E\xi_n = \mu \pi_\mu + \lambda \pi_\lambda$ for all $n$.

Assume $\bs{X}\coloneqq (X_0, X_1,\dots)$ is a sequence of $S$-valued random variables   and that $\bs\xi$ is a product disintegration of $\bs X$. Then we have, for $x\in \{-1,1\}$, that $\Prob(X_n = x) = \E\xi_n(x) = \mu(x)\pi_\mu + \lambda(x)\pi_\lambda$. We also have, for $x_0,x_1\in \{-1,1\}$,
\begin{align}\label{eq:markov-switching}
\begin{split}
\Prob(X_0 = x_0, X_1 = x_1) &= \E \xi_0(x_0)\xi_1(x_1)\\
&= \mu(x_0)\mu(x_1)\pi_{\mu}Q_{\mu\mu} + \mu(x_0)\lambda(x_1)\pi_{\mu}Q_{\mu\lambda}\\
&\qquad +\lambda(x_0)\mu(x_1)\pi_\lambda Q_{\lambda\mu} + \lambda(x_0)\lambda(x_1)\pi_\lambda Q_{\lambda\lambda}
\end{split}
\end{align}
This shows that in general it may be difficult to compute the finite-dimensional distributions of the process $(X_0, X_1,\dots)$ --- although this process inherits stationarity from $\bs\xi$. Also, an easy check tells us that generally speaking $\bs X$ is not a Markov chain. 

Nevertheless, assuming $Q$ is irreducible and positive recurrent (i.e., $\pi_\mu\notin\lbrace0,1\rbrace$), we have by the ergodic theorem for Markov chains, that
\begin{equation}\label{eq:markov-chain-ergodic-theorem}
\lim_{n\to\infty}\frac{1}{n}\sum_{k=0}^{n-1}h\circ\xi_k = \pi_\mu h(\mu) + \pi_\lambda h(\lambda) = \E h\circ\xi_0,\qquad{\mathrm{a.s}},
\end{equation}
for any bounded $h\colon M'\to\R$. 
Now let $f\colon S\to \R$ and consider the particular case where $h\circ \xi := \xi (f)$. Equation \ref{eq:markov-chain-ergodic-theorem} becomes
\begin{equation}\label{eq:markov-chain-ergodic-theorem-nova}
\lim_{n\to\infty}\frac{1}{n}\sum_{k=0}^{n-1} \xi_k (f) = \pi_\mu \mu(f) + \pi_\lambda \lambda(f) = 
\E \xi_0(f) ,\qquad{\mathrm{a.s}}.
\end{equation}
Therefore, using  Theorem~\ref{thm:prop1} and then \eqref{eq:markov-chain-ergodic-theorem-nova}, we have that
\begin{align*}
\lim_{n\to\infty}\frac{1}{n}\sum_{k=0}^{n-1}f\circ X_k & = \E \xi_0 (f) \\ 
& = \mu(f) \Prob[\xi_0=\mu]+\lambda(f)\Prob[\xi_0=\lambda] \\ 
& = \left( f(1)\mu(1) + f(-1)\mu(-1)\right)\pi_\mu 
+ \left( f(1)\lambda(1) + f(-1)\lambda(-1)\right)\pi_\lambda \\ 
& = f(1)\big(\mu(1)\pi_\mu +\lambda(1)\pi_\lambda\big) + f(-1)\big(\mu(-1)\pi_\mu +\lambda(-1)\pi_\lambda\big)
\end{align*}
holds almost surely.
In particular this is true with $f = \I_{\lbrace1\rbrace}$; thus, even though the `ups and downs' of $\bs X$ are governed by a law which can be rather complicated (as one suspects by inspecting equation \eqref{eq:markov-switching}), we can still estimate the overall (unconditional) probability of, say, the expansion regime by computing the proportion of ups in a sample $(X_0,\dots,X_n)$:
\[  
\lim_{n\to\infty}\frac{1}{n}\sum_{k=0}^{n-1} \I_{\lbrace1\rbrace}(X_k) = \mu(1)\pi_\mu +\lambda(1)\pi_\lambda.
\]
\end{example}

\begin{example}
Suppose $\bs\vartheta = (\vartheta_0,\vartheta_1,\dots)$ is a submartingale, with $\mathrm{range}(\vartheta_n)\subseteq[0,1]$ for all $n$. By the Martingale Convergence Theorem, there exists a random variable $\vartheta_\infty$ such that $\lim \vartheta_n = \vartheta_\infty$ almost surely (thus, we can assume without loss of generality that $0\leq\vartheta_\infty\leq1$). Furthermore, let $S\coloneqq\{0,1\}$ and, for $n\in\N$, let $\xi_n$ denote the random probability measure on $S$ defined via $\xi_n(\{1\}) = \vartheta_n$, and $\xi_n(\{0\}) = 1-\vartheta_n$. We have $\xi_n(\I_{\lbrace1\rbrace}) = \vartheta_n \to\vartheta_\infty$ a.s. Assume further that $\bs\xi = (\xi_0, \xi_1,\dots)$ is a product disintegration of a sequence $\bs X = (X_0, X_1,\dots)$ of random variables with values in $S$. Using Theorem~\ref{thm:prop1} we have 
\[ \lim_{n \to \infty}  \frac{1}{n}\sum_{i=0}^{n-1} \I_{\lbrace1\rbrace}(X_i) = 
\lim_{n \to \infty}  \frac{1}{n}\sum_{i=0}^{n-1} \xi_i(\I_{\lbrace1\rbrace})
= \lim_{n \to \infty} \frac{1}{n}\sum_{i=0}^{n-1}\vartheta_i=\vartheta_\infty \quad \text{a.s.}
\]
which means, the proportion of 1's in $(X_0,\dots,X_n)$ approaches $\vartheta_\infty$ with probability one.

To illustrate, let $(U_n\colon n\geq1)$ be a sequence of independent and identically distributed Uniform$[0,1]$ random variables. Let $\vartheta_0 = U_0/2$ and, for $n\geq1$, define $\vartheta_{n} \coloneqq \vartheta_{n-1} + 2^{-(n+1)}U_{n}$. Figure \ref{fig:submartingale} displays, in blue, a simulated sample path of the submartingale $(\vartheta_n\colon n\in\N)$ up to $n=20$. The $\circ$'s represent the successive outcomes of the coin throws (where the probability of `heads' in the $n$th throw is $\vartheta_n$). In purple are displayed the sample path of the means $(n^{-1}S_n\colon n\in \N)$, where $S_n$ is the partial sum $\sum_{i=0}^{n-1}X_n$. In this model, even if we only observe the outcomes of the coin throws, we can still estimate the value of $\vartheta_\infty$: all we need to do is to compute the proportion of heads in $X_0,\dots,X_n$, with $n$ large.
\begin{figure}
\begin{center}
\includegraphics{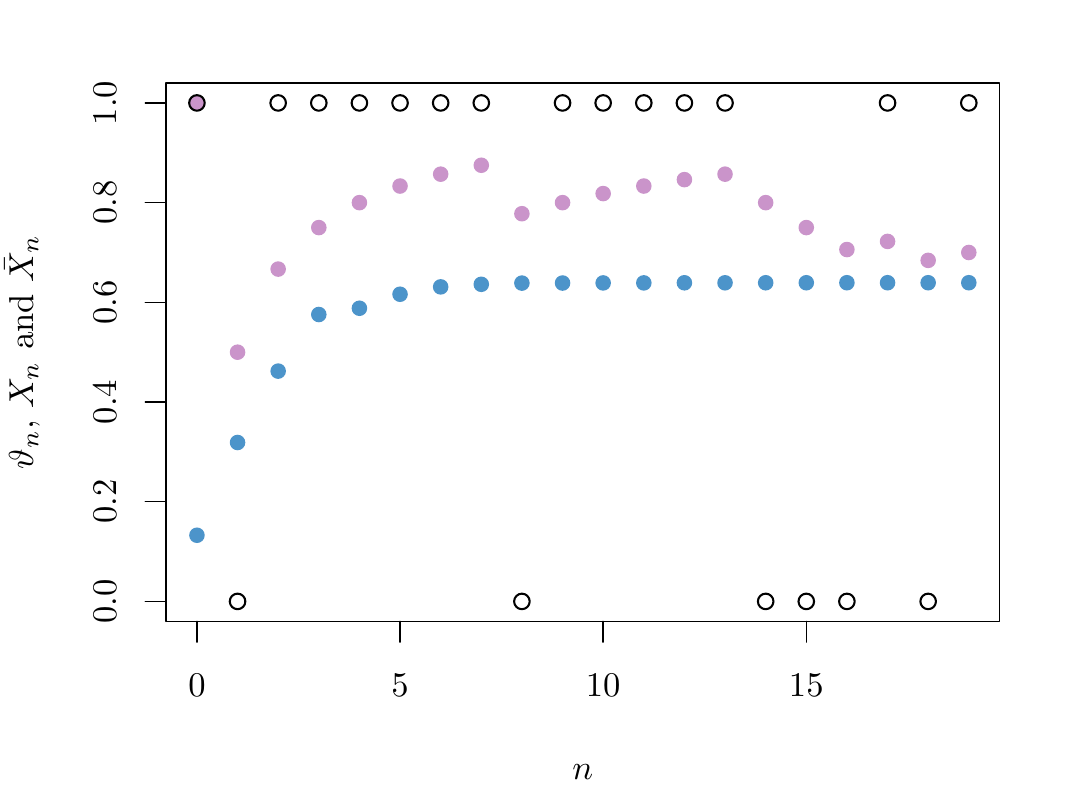}
\caption{A sample path of the submartingale $(\vartheta_n)$.}\label{fig:submartingale}
\end{center}
\end{figure}
\end{example}

\begin{example}
We now show that a certain class of stochastic volatility models can be accommodated into our framework of product disintegrations. Stochastic volatility models are widely used in the financial econometrics literature, as they provide a parsimonious approach for describing the volatility dynamics of a financial asset's return --- see \cite{shephard2009overview} and \cite{davis2009svmodels} and references therein for an overview. A basic specification of the model\footnote{Which can be relaxed by putting $g\circ H_t$ in place of $\ee^{H_t/2}$, and allowing $\bs H$ to evolve according to more flexible dynamics.} is as follows: let $\bs Z = (Z_t\colon t\in\Z)$ and $\bs W = (W_t\colon t\in \Z)$ be centered iid sequences, independent from one another, and define $\bs X$ and $\bs H$ via the stochastic difference equations
\[
X_t = \ee^{H_t/2}Z_t,\quad t\in\N,
\quad\mbox{and}\quad
H_t = \alpha + \beta H_{t-1} + W_t,\quad t\geq 1,
\]
where $\alpha$ and $\beta$ are real constants and where $H_0$ follows some prescribed distribution. The random variable $X_t$ is interpreted as the return (log-price variation) on a given financial asset at date $t$, and the $H_t$'s are latent (i.e, unobservable) random variables that conduct the volatility of the process $\bs X$. Usually this process is modelled with Gaussian innovations, that is, with $W_t$ and $Z_t$ normally distributed for all $t$. In this case the random variables $X_t$ are supported on the whole real line, so we need to consider other distributions for $\bs Z$ and $\bs W$ if we want to ensure that the $X_t$'s are compactly supported. 

Our objective is to show how Theorem~\ref{thm:prop1} can be used to estimate certain functionals of the latent volatility process $\bs H$ in terms of the observed return process $\bs X$. To begin with, notice that if $|\beta|<1$ and if $H_0$ is defined via the series $H_0 \coloneqq {(1-\beta)}^{-1} \alpha+\sum_{k=0}^{\infty} \beta^k W_{-k}$, then $\bs H$ (and $\bs X$) is strictly stationary and ergodic, in which case we have that
\begin{equation}
\lim_{n\to\infty} \frac1n\sum_{t=0}^{n-1}
g\circ H_t = \E g\circ H_0
\end{equation}
almost surely, for any $\Prob_{H_0}$-integrable $g\colon S_H\to\R$, where we write $S_H \coloneqq \mathrm{supp}\, H_0$. Also, notice that, by construction, we have for all $n$, all measurable $A_0,\dots,A_n\subseteq S \coloneqq \mathrm{supp}\,X_0$ and all $\bs h = (h_0,h_1,\dots)\in S_H^{\N}$,
\begin{align}
\nonumber\Prob\big(X_0 \in A_0,\dots,X_n\in A_n\,|\,\bs H = \bs h\big)  &=
\Prob\big(\ee^{H_0/2}Z_0\in A_0,\dots, \ee^{H_n/2}Z_n\in A_n\given \bs H = \bs h\big)\\
\nonumber&\overset{(*)}=\Prob\big(\ee^{h_0/2}Z_0\in A_0,\dots, \ee^{h_n/2}Z_n\in A_n\given \bs H = \bs h\big)\\
\nonumber&\overset{(**)}=\Prob\big(\ee^{h_0/2}Z_0\in A_0,\dots, \ee^{h_n/2}Z_n\in A_n\big)\\
\nonumber&= \prod_{t=0}^n
\Prob\big(\ee^{h_t/2}Z_t\in A_t\big)\\
&\overset{(***)}= \prod_{t=0}^n
\Prob\big(X_t\in A_t\,|\,\bs H = \bs h\big).\label{eq:product-disintegraion-H}
\end{align}
Where $(*)$ is yielded by the substitution principle, $(**)$ follows from the fact that $\bs Z$ and $\bs H$ are independent (as $\bs H$ only depends on $\bs W$), and $(***)$ is just a matter of repeating the previous steps. A reasoning similar to the one used in the proof of Lemma~\ref{thm:product-measure-disintegration} then tells us that $\Prob(\bs X\in\cdot\,|\,\bs H)_\omega$ is a product measure on $S^{\N}$ for almost all $\omega$. Also, notice that in particular we have that $\Prob(X_t\in A\given \bs H = \bs h) = \Prob(\ee^{h_t/2}Z_t\in A)$ for all $t$. In fact, let $\varphi\colon S_H\to M_1(S)$ be defined via $\varphi(h,A):= \Prob(\ee^{h/2}Z_0\in A)$, for $h\in S_H$ and measurable $A\subseteq S$, where we write $\varphi(h,A)$ in place of $\varphi(h)(A)$ for convenience. Since the $Z_t$'s are identically distributed, we have in particular that $\varphi(h,A) = \Prob(\ee^{h/2}Z_t\in A)$ for all $t$. The preceding derivations now allow us to conclude that
\begin{equation}\label{eq:product-disintegraion-H-2}
\varphi(H_t(\omega),A) = \Prob(X_t\in A\given H_t)_\omega = \Prob(X_t\in A\given \bs H)_\omega.
\end{equation}

We are now in place to introduce a product disintegration of $\bs X$, by defining $\xi_t^\omega(A):= \Prob(X_t\in A\given \bs H)_\omega$ for measurable $A\subseteq S$, $t\in\Z$ and $\omega\in \Omega$. To see that $\bs{\xi} = (\xi_0,\xi_1,\dots)$ is indeed a product disintegration of $\bs X$, first notice that $\xi_0(A_0)\cdots\xi_n(A_n)$ is  $\sigma(\bs \xi)$-measurable for every $n$ and every $(n+1)$-tuple $A_0,\dots,A_n$ of measurable subsetes of $S$. Moreover, defining $\psi\colon S_H^\N\to M_1(S)^\N$ via $\psi(h_0,h_1,\dots) = (\varphi(h_0), \varphi(h_1),\dots)$, we obtain, by equations~\eqref{eq:product-disintegraion-H} and \eqref{eq:product-disintegraion-H-2},
\begin{align*}
\E\big(\xi_0(A_0)\cdots\xi_n(A_n)\I_{[\bs\xi\in\bs B]}\big) &= \E\big(\varphi(H_0,A_0)\cdots\varphi(H_n,A_n)\I_{[\bs H \in \psi^{-1}(\bs B)]}\big)\\
&= \Prob\big(X_0\in A_0,\dots, X_n\in A_n, \bs H\in \psi^{-1}(\bs B)\big)\\
&= \Prob\big(X_0\in A_0,\dots, X_n\in A_n, \bs\xi\in \bs B\big),
\end{align*}
whence $\Prob(X_0\in A_0,\dots,X_n\in A_n\,|\,\bs\xi) = \xi_0(A_0)\cdots\xi_n(A_n)$, and then Lemma~\ref{thm:product-measure-disintegration} tells us that $\bs\xi$ is --- voilà --- a product disintegration of $\bs X$.

Now, since $\varphi$ is continuous and one-to-one, we have that $\varphi$ is a homeomorphism from $S_H$ onto its range whenever $S_H$ is compact (in particular, $\mathrm{range}(\varphi)$ is compact, hence measurable, in $M_1(S)$). Also, as $\xi_t = \varphi\circ H_t$ for all $t$, we have that $H_t = \varphi^{-1}\circ \xi_t$ is well defined. Suppose now that $f\colon S\to\R$ is a given continuous function. We have \[\xi_t^\omega(f) = \int_S f(x)\,\xi_t^\omega(\dd x) = \int_S f(x)\,\varphi(H_t(\omega),\dd x) =: g(H_t(\omega))\] and, as $\bs H$ is ergodic, it holds that
\(
\lim_{n\to\infty} n^{-1} \sum_{t=0}^{n-1} \xi_t(f) = \E g\circ H_0,
\)
where we know that the expectation is well defined, as 
\(
\E |g\circ H_0| \leq \E \left(\int_S |f(x)|\,\xi_0(\dd  x) \right)< \infty,
\)
with the expected value given by $\E g\circ H_0 = \int_S f(x)\,\Prob_{X_0}(\dd x)$.
We can now apply Theorem~\ref{thm:prop1} to see that
\[
\E g\circ H_0 = \lim_{n\to\infty} n^{-1} \sum_{t=0}^{n-1} f\circ X_t.
\]
The conclusion is that, for suitable $g$ of the form $g(h) = \int_S f(x)\,\varphi(h,\dd x)$, we can estimate $\E g\circ H_0$ by the data $(X_0, X_1, \dots, X_n)$ as long as $n$ is large enough, even if we cannot observe $\bs H$. Of course, this follows from ergodicity of $\bs X$, but it is interesting anyway to arrive at this result from an alternate perspective; moreover, one can use Hoeffding type inequalities as in Example~\ref{example:concentration} to easily derive a rate of convergence for sample means of $\bs X$ based on the rate of convergence of sample means of $\bs H$.
\end{example}

\section{Concluding remarks}
In this paper we prove that a sequence of Bernoulli$(p)$ random variables satisfies the strong law of large number if and only if the sequence is conditionally independent, where the conditioning is on a sequence of $[0,1]$-valued random variables, whose corresponding sequence of sample means converges almost surely. As a byproduct, we introduce the concept of a \emph{product disintegration}, which generalizes exchangeability. Some applications of the concept are illustrated in Section~\ref{sec:examples}. Further applications of product disintegrations and of Theorem~\ref{thm:prop1} appear as a possible path to be pursued in future work.

\medskip
\noindent \textbf{A road not taken.} At some point, during the development of the present paper, we delved into the possibility of translating our approach to the language of Ergodic Theory. This proved more difficult than we first thought, but we did come up with a conjecture: consider the \emph{left--shift operator} $T$ acting on $S^\N$, given by
\(
(T\bs x)_i = x_{i+1}
\)
for $\bs x = (x_0,x_1,\dots)\in S^\N$, and define $\tilde{T}\colon M_1\left(S\right)^\N\rightarrow M_1\left(S\right)^\N$ analogously. Recall that $\rho(\bs\lambda)\coloneqq \prod_{i\in\N}\lambda_i$ for $\bs\lambda\in M_1(S)^\N$.

\begin{conjecture}\label{thm:ergodicity-characterization}
Let $S$ be a compact metric space. A $T$-invariant measure $q\in M_1(S^\N)$ is $T$--ergodic if and only if there exists a $\tilde{T}$--ergodic measure $Q\in M_1(M_1\left(S\right)^\N)$ such that $q = \int \rho\left(\bs{\lambda}\right) Q\left(\dd \bs{\lambda}\right)$.
\end{conjecture}


\begin{appendix}
\section*{Auxiliary results}
Given a topological space $S$, we will write $B\leq S$ to mean that $B$ belongs to the $\sigma$-field generated by the topology of $S$ (i.e., that $B$ is a Borel subset $S$).

\subsection{Spaces of measures}\label{sec:spaces-of-measures}
For a compact metric space $S$ endowed with its Borel $\sigma$-field, let $\mbox{Meas}_b(S)$ denote the set of measurable, bounded maps $f\colon S\to\R$, let $C(S) \subseteq \mbox{Meas}_b(S)$ denote the set of continuous maps from $S$ to $\R$, and let $M(S)$ denote the set of finite Borel measures on $S$. 
As in the main text, $M_1(S)\subseteq M(S)$ denotes the set of Borel probability measures on $S$. For $f\in \mbox{Meas}_b(S)$, we define the evaluation map $\hat{f}\colon M(S)\to\R$ by
\(
\hat{f}(\mu) \coloneqq \int f(x)\,\mu(\dd x),
\) for $\mu\in M(S)$.

There are a few manners through which one can introduce a $\sigma$-field on $M(S)$ (and, a fortiori, on $M_1(S)$). The most commonly adopted approach is to consider in $M(S)$ the weak* topology relative to $C(S)$ (in conventional probabilistic jargon, this is simply called the \emph{weak} topology), that is, the coarsest topology on $M(S)$ for which, for every $f \in C(S)$, its evaluation map $\hat{f}$ is continuous. The following theorem presents some very useful results. Item 2 is related to Prokhorov's compactness criterion, but is not restricted to probability measures. The last three items show that, if the aim is to obtain a $\sigma$-field in $M(S)$, there is no need for topological considerations (on $M(S)$).

\begin{theorem}
\label{thm:KAL-A2-3}
Let $S$ be a compact metric space. Then
\begin{enumerate}
\item $M(S)$ is Polish (i.e. is separable and admits a complete metrization) in the weak* topology.
\item\label{thm:KAL-A2-3-itm-2} A set $K\subseteq M(S)$ is weakly* relatively compact if and only if 
\(
\sup_{\mu\in K}\hat{f}(\mu) < \infty
\)
for all nonnegative $f\in C(S)$.
\item The Borel $\sigma$-field relative to the weak* topology on $M(S)$ coincides with the $\sigma$-field $\sigma(\mathscr{C})$, where $\mathscr{C}$ can be taken as any one of the following classes:
\begin{enumerate}
\item[i.] $\mathscr{C}=\{\hat{f}\colon\, f\in C(S),\,f\geq 0\}.$
\item[ii.] $ \mathscr{C}=\{\hat{f}\colon\, f = \mathbb{I}_B,\,B\leq S\}.$
\item[iii.] $\mathscr{C}=\{\hat{f}\colon\, f \in \mathrm{Meas}_b(S)\}.$
\end{enumerate}
\end{enumerate}
\end{theorem}

\begin{remark}
In summary, item (3) above says the following: if we write $\tau(\hat{f}\colon\,f\in C(S))$ for the topology on $M_1(S)$ generated by the mappings $(\hat{f}\colon\,f\in C(S))$ (that is, the weak* topology), then
\(
\sigma\left(\tau\big(\hat{f}\colon\,f\in C(S)\big)\right) = \sigma\left(\hat{f}\colon\,f\in C(S)\right),
\)
etc.
\end{remark}

\begin{proof}
For the first two items, and sub-items i. and ii. of the last item, see Theorem~A2.3 in \cite{kallenberg2002foundations}. The proof will be complete once we establish the identity
\begin{equation*}
\sigma\{\hat{f}\colon\, f \in \mathrm{Meas}_b(S)\} = \sigma\{\hat{f}\colon\, f = \mathbb{I}_B,\,B\leq S\}.
\end{equation*}
Clearly the inclusion $\sigma\{\hat{f}\colon\, f \in \mathrm{Meas}_b(S)\} \supseteq \sigma\{\hat{f}\colon\, f = \mathbb{I}_B,\,B\leq S\}$ holds.

For the converse inclusion, it is enough to show that, for every $g \in \mathrm{Meas}_b(S)$, one has $\hat{g} \in \sigma\{\hat{f}\colon\, f = \I_B,\,B\leq S\} =: \Tilde{\mathcal{B}}$. If $g = \I_B$ for some $B\leq S$, then clearly $\hat{g} \in \Tilde{\mathcal{B}}$. If $g$ is simple, with standard representation
\(
g(x) = \sum_{j=1}^n a_j \mathbb{I}_{A_j}(x),
\)
then $\hat{g}(\lambda) = \sum_{j=1}^n a_j \hat{g}_j(\lambda)$, where $g_j = \mathbb{I}_{A_j}$. Thus, $\hat{g} \in \Tilde{\mathcal{B}}$ as it is a linear combination of elements of $\Tilde{\mathcal{B}}$. For the general $g\in \mathrm{Meas}_b(S)$, let $(g_n)$ be a sequence of simple functions with $|g_n|\leq |g|$ and $g_n\to g$. Then the Dominated Convergence Theorem gives $\hat{g}(\lambda) = \lim \hat{g}_n(\lambda)$ and hence $\hat{g} \in \Tilde{\mathcal{B}}$, which concludes the result.
\end{proof}

Since $M_1(S) = \{\mu\in M(S)\colon \hat{f}(\mu) = 1\} = \hat{f}^{-1}(\{1\})$, with ${f} = \mathbb{I}_S\in C(S)$, clearly $M_1(S)$ is a weakly* closed (hence measurable) subset of $M(S)$. By item~\ref{thm:KAL-A2-3-itm-2} in Theorem~\ref{thm:KAL-A2-3}, $M_1(S)$ is weakly* compact. Indeed, more can be said: $M_1(S)$ is a compact metrizable space. Usually, this fact is stated in terms of the so called Lévy-Prokhorov metric, which works for quite general $S$ but suffers from a ``lack of interpretability''. Conveniently, when $S$ is compact there is an equivalent metric generating the weak* topology, given by
\(
d(\mu,\nu)=\sum_{k\geq 1} 2^{-k} \left| \hat{f}_k(\mu)-\hat{f}_k(\nu)\right|,
\)
where $\left\{ f_k \right\}_{k \geq 1}$ is a dense and countable subset of the unit ball in $C(S)$. The following result is an immediate corollary to Theorem 8.3.2 in \cite{bogachev2007measure}:

\begin{theorem}\label{thm:bogachev-8-3-2}
The weak* topology on $M_1(S)$ is metrizable.
\end{theorem}

\subsection{Random probability measures}

\begin{definition}
A \emph{random probability measure on $S$} is defined to be a Borel measurable map $ \xi\colon\Omega\to M_1(S) $. We shall denote the value of a random probability measure $\xi$ at the point $\omega$ by $\xi^\omega$ and, for a Borel subset $B\subseteq S$, we will use the notation $\xi^\omega(B)$ and $\xi(\omega, B)$ undistinguishedly. The latter notation is justified in Theorem~\ref{thm:random-measure-equivalence} below.
\end{definition}

\begin{lemma}[measurability of $\xi$]\label{thm:xi-measurability}
A map $\xi\colon \Omega\to M_1(S)$ is measurable if and only if the map $\omega\mapsto \int f\,\dd\xi^\omega = \hat{f}\circ\xi (\omega)$ is a random variable for every $f\in \mathscr{C}$, where $\mathscr{C}$ can be taken as any one of the sets $C(S)$, $\mathrm{Meas}_b(S)$ or $\{\mathbb{I}_B\colon\,B\leq S\}$.
\end{lemma}

\begin{proof}
By Theorem~\ref{thm:KAL-A2-3}, the Borel $\sigma$-field on $M_1(S)$ is given by
\begin{equation*}
\sigma\{\hat{f}\colon f\in C(S)\} = \sigma\{\hat{f}\colon f\in \mathrm{Meas}_b(S)\} = \sigma\{\hat{f}\colon f = \mathbb{I}_B,\,B\leq S\}.
\end{equation*}
The `only if' part follows immediately. For the `if' part, notice that $\sigma\{\hat{f}\colon f\in \mathscr{C}\}$ is the smallest $\sigma$-field containing the sets $\hat{f}^{-1}(E)$, with $f\in\mathscr{C}$ and $E\leq\R$. Now $\hat{f}\circ\xi$ is measurable for every $f\in\mathscr{C}$ iff $(\hat{f}\circ\xi)^{-1}(E)\in\mathscr{F}$ for every $f\in\mathscr{C}$ and every $E\leq \R$ iff $\xi^{-1}(G)\in\mathscr{F}$ for every $G$ of the form $\hat{f}^{-1}(E)$ with $f\in\mathscr{C}$ and $E\leq\R$. Since the class of such $G$ generates the Borel $\sigma$-field on $M_1(S)$, the result follows. 
\end{proof}

\begin{theorem}[existence of Baricenter]\label{thm:baricenter-existence}
Let $(\Omega,\mathscr{F},\Prob)$ be a probability space, $S$ a compact metric space, and let $\xi$ be a random probability measure on $S$. Then there exists a unique element $\bar{\mu} \in M_1(S)$ such that the equality
\(
\int_S f(x)\,\bar{\mu}(\dd x) = \int_\Omega\,\int_S f(x)\,\xi^\omega(\dd x)\,\Prob(\dd \omega)
\)
holds for all $f\in C(S)$.
\end{theorem}

\begin{proof}
Let $\varphi\colon C(S)\to \R$ be defined by
\(
\varphi(f)\coloneqq \int_\Omega\,\int_S f(x)\,\xi^\omega(\dd x)\,\Prob(\dd \omega).
\)
Clearly $\varphi(f)\geq 0$ if $f\geq 0$, $\varphi(\alpha f+g) = \alpha\varphi(f)+\varphi(g)$, and $\varphi(1)=1$. Thus, by the Riesz-Markov Theorem, there is an element $\bar{\mu}\in M_1(S)$ such that the stated equality holds.
\end{proof}

\begin{definition}
The unique element $\bar{\mu}$ yielded by Theorem~\ref{thm:baricenter-existence} is called the \emph{baricenter of $\xi$} (analogously: the \emph{$\Prob$-expectation of $\xi$}; analogously: the \emph{baricenter of $\Prob_\xi$}). Notation:
\begin{equation*}
\bar{\mu} =: \int \xi\,\dd \Prob =: \E_\Prob\xi.
\end{equation*}
\end{definition}
We also write simply $\E\xi$ in place of $\E_\Prob\xi$ when $\Prob$ is understood from context.

\begin{lemma}\label{thm:random-measure-basics}
Let $\xi$ be a random probability measure on $S$ and let $\E\xi$ be its baricenter. Then
\begin{enumerate}
\item (commutativity) For each measurable subset $B\subseteq S$, the equality $\E\xi\left(B\right) = \E\left(\xi\left(B\right)\right)$ holds; \label{thm:random-measure-basics-item-ii}
\item (maximal support) there exists a set $\Omega_0$ with $\Prob\left(\Omega_0\right)=0$ such that, for $\omega\notin\Omega_0$, the relation $\supp\xi^\omega\subseteq\supp\E\xi$ holds.
\end{enumerate}
\end{lemma}

\begin{proof}
For the first item, let $\lambda(B)\coloneqq \E(\xi(B))$, $B\subseteq S$ measurable. Clearly we have $\lambda(B) \geq 0$ and $\lambda(\Omega) = 1$. Moreover, if $(B_j)$ is a sequence of measurable subsets of $S$ which are pairwise disjoint such that $B = \bigcup_{j=1}^\infty B_j$, then for each $\omega$ we have
\(
\xi^\omega(B) = \lim_{n\to\infty}\sum_{j=1}^n \xi^\omega(B_j)\leq 1.
\)
Thus, by the Dominated Convergence Theorem (DCT), we have $\lambda(B) = \sum_{j=1}^\infty \lambda(B_j)$. Therefore $\lambda$ is a probability measure on $S$. Now let $K\subseteq S$ be closed and let $(f_n)$ be a sequence of continuous functions on $S$ such that $1\geq f_n(x)\rightarrow \I_K(x)$, $x\in S$. On the one hand we have $\E\xi(K) = \lim_{n\to\infty} \int f_n(x)\, \E\xi(\dd x)$, by DCT. On the other hand, for each $\omega$ it holds that
\begin{equation*}
0\leq \xi^\omega(K) = \lim_{n\to\infty}\int f_n(x)\,\xi^\omega(\dd x)\leq \Vert f_n\Vert_\infty \leq 1,
\end{equation*}
again by DCT. Applying the DCT once more yields
\begin{equation*}
\lambda(K) = \lim_{n\to\infty}\int \int f_n(x)\,\xi^\omega(\dd x)\, \Prob(\dd\omega) = \lim_{n\to\infty} \int f_n(x)\, \E\xi(\dd x) = \E\xi(K),
\end{equation*}
where the second equality follows from the definition of the baricenter. Thus, $\E\xi$ and $\lambda$ are measures on $S$ whose values on closed sets coincide, and this implies $\E\xi = \lambda$, as asserted.

For the second item, let $U\coloneqq S\setminus\supp\left(\E\xi\right)$. Then $\xi\left(U\right)\geq 0$ and $\E\left(\xi\left(U\right)\right) = \E\xi\left(U\right) = 0$, by item~\ref{thm:random-measure-basics-item-ii}. Hence $\xi\left(U\right) = 0$ almost surely.
\end{proof}

\subsection{Probability Kernels}
\begin{definition}[see \citep{kallenberg2002foundations}, page 20]
Given two measurable spaces $(\Omega, \mathscr{F})$ and $(S,\mathscr{S})$, a map $\xi\colon \Omega\times \mathscr{S}\to \R$ is said to be a \emph{probability kernel from $(\Omega,\mathscr{F})$ to $(S,\mathscr{S})$} iff
\begin{enumerate}
\item[D1] For each $\omega\in \Omega$, the map $B\to \xi(\omega, B)$ is a probability measure on $S$.
\item[D2] For each $B\leq S$, the map $\omega\mapsto \xi(\omega,B)$ is $\mathscr{F}$-measurable.
\end{enumerate}
\end{definition}

Kernels play an important role in probability theory, appearing in many forms, for example random measures, conditional distributions, Markov transition functions, and potentials \citep{kallenberg2002foundations}. Indeed, in many circumstances, one feels more comfortable working with probability kernels instead of random probability measures as defined above, given the prevalence of the former concept in the literature. The following result connects the two concepts, showing that they are indeed equivalent: 
\begin{theorem}
\label{thm:random-measure-equivalence}
Fix two measurable spaces $(\Omega, \mathscr{F})$ and $(S,\mathscr{S})$, and assume $\mathscr{S} = \sigma(\mathscr{C})$ for some $\pi$-system $\mathscr{C}$. Let $\xi\colon \Omega\times\mathscr{S}\to \R$ be such that $\xi(\omega,\cdot)$ is a probability measure on $S$, for every $\omega\in\Omega$. Then the following conditions are equivalent:
\begin{enumerate}
\item $\xi$ is a probability kernel from $(\Omega,\mathscr{F})$ to $(S,\mathscr{S})$.
\item $\omega\mapsto \xi(\omega, \cdot)$ is an $\mathscr{F}$-measurable mapping from $\Omega$ to $M_1(S)$.
\item $\omega\mapsto \xi(\omega,E)$ is an $\mathscr{F}$-measurable mapping from $\Omega$ to $[0,1]$ for every $E\in \mathscr{C}$.
\end{enumerate}
In particular, the above equivalences hold with $\mathscr{C} = \mathscr{F}$.
\end{theorem}
\begin{proof}
This is just a restatement of Lemma~1.40 in \cite{kallenberg2002foundations}, by noticing that the Borel $\sigma$-field on $M_1(S)$ coincides with $\sigma(\hat{f}\colon\, f=\I_B,\,B\leq S)$ as ensured by Theorem~\ref{thm:KAL-A2-3}.
\end{proof}

\begin{definition}
A kernel $\xi$ from $(\Omega,\mathscr{F}_0)$ to $(S,\mathscr{S})$ is said to be a \emph{regular conditional distribution} of a random variable $X\colon\Omega\to S$ given a $\sigma$-field $\mathscr{F}_0 \subseteq \mathscr{F}$ iff the equality
\(
\int_F \xi(\omega, B)\,\Prob(\dd\omega) = \Prob([X\in B]\cap F)\)
holds for all $B\leq S$ and $F\in\mathscr{F}_0$. In particular, for each $B\leq S$ the random variable $\xi(\cdot,B)$ is a version of $\Prob(X\in B\,|\,\mathscr{F}_0)$.
\end{definition}

\begin{theorem}[Regular conditional distribution --- \citep{kallenberg2002foundations}, Theorem~6.3]\label{thm:regular-conditional-distribution}
Let $(S,\mathscr{S})$ and $(T,\mathscr{T})$ be measurable spaces, and let ${X}$ and ${\xi}$ be random variables taking values in $S$ and $T$ respectively. Assume further that $S$ is Borel. Then there exists a probability kernel $\eta$ from $T$ to $S$ such that
\(
\Prob({X}\in B\given{\xi})_\omega = \eta({\xi}(\omega),B)
\)
for all $B\in\mathscr{S}$ and all $\omega$ in a set $\Omega^*\subseteq\Omega$ with $\Prob(\Omega^*)=1$. Moreover, $\eta$ is unique almost everywhere-$\Prob_{{\xi}}$.
\end{theorem}
\begin{remark}
In the conditions of the above Theorem, one can introduce a probability kernel $\eta'$ from $(\Omega,\sigma(\xi))$ to $(S,\mathscr{S})$ by putting $\eta'(\omega,B)\coloneqq \eta(\xi(\omega),B)$. Also, if $\xi$ is the identity map and $(T,\mathscr{T}) = (\Omega, \mathscr{F}_0)$, where $\mathscr{F}_0 \leq \mathscr{F}$, then automatically $\eta$ is a kernel from $\Omega$ to $S$ which is a regular version of $\Prob(X\in \cdot\,|\,\mathscr{F}_0)$.
\end{remark}

\subsection{Product spaces}

\begin{lemma}[Product and Borel $\sigma$-fields --- \citep{kallenberg2002foundations}, Lemma~1.2]\label{thm:product-topology}
Let $S$ have topology $\tau$ and let $\mathscr{S}\coloneqq\sigma\left(\tau\right)$ be the Borel $\sigma$-field on $S$. Let $\tau^\N$ be the product topology on $S^\N$ and let $\mathscr{S}^\N$ be the product (cylindrical) $\sigma$--field on $S^\N$. If $S$ is metrizable and separable, then $\sigma(\tau^\N)=\mathscr{S}^\N$, that is, $\mathscr{S}^\N$ is the Borel $\sigma$-field on $S^\N$.
\end{lemma}

\begin{corollary}\label{thm:subbasis}
Let $S$ be a separable metric space with topology $\tau$, and let $\mathscr{B}$ be a countable basis for $\tau$ which is stable under finite intersections. For each $n\in \N$ and each $B_0,\dots,B_n\in \mathscr{B}$, define
\begin{equation}\label{eq:subbasis}
\mathscr{C}(n;B_0,\dots,B_n)\coloneqq B_0\times \cdots \times B_n\times S\times\cdots \subseteq S^\N.
\end{equation}
Let $\mathfrak{C}$ denote the collection of all sets of the form \eqref{eq:subbasis}. Then $\mathfrak{C}$ is a countable $\pi$-system which generates the Borel $\sigma$-field on $S^\N$.
\end{corollary}

\begin{remark}
The set $\mathscr{B}$ above can be obtained as follows: let $D$ be a countable, dense subset of $S$, and let $\mathscr{D}$ be the collection of all balls with centers in $D$ and rational radii. Now let $\mathscr{B}_n$, $n\geq1$, be the collection formed by all intersections of $n$ elements of $\mathscr{D}$, that is, $\bs B\in\mathscr{B}_n$ iff there exist $x_1,\dots,x_n\in D$ and $r_1,\dots,r_n\in\mathbb Q$ such that $\bs B = \bigcap\nolimits_{i=1}
^n\mathrm{ball}(x_i; r_i)$. Clearly, each $\mathscr{B}_n$ is countable. Now let $\mathscr{B}:=\bigcup\nolimits_{n\geq1} \mathscr{B}_n$.
\end{remark}

\begin{proof}[Proof of Corollary~\ref{thm:subbasis}] We begin by proving that $\mathfrak{C}$ is indeed a $\pi$-system. Clearly, $\mathfrak{C}$ is non-empty. Now, let $A_0, \dots, A_m, B_0, \dots, B_n \in \mathscr{B}$ and consider $\mathscr{C}(m; A_0, \dots, A_m), \mathscr{C}(n; B_0, \dots, B_n)$. Without loss of generality, suppose $n \geq m$. Then 
\begin{align*}
    \mathscr{C}(m; A_0, \dots, A_m) \cap \mathscr{C}(n; B_0, \dots, B_n) &= A_0 \cap B_0 \times \dots \times A_m \cap B_m \times \dots \times B_n \times S \times \dots\\
    &= \mathscr{C}(n; A_0\cap B_0, \dots, A_m\cap B_m, \dots, B_n ).
\end{align*}
Since $\mathscr{B}$ is stable under finite intersections, $A_i \cap B_i \in \mathscr{B}$, for each $i \in \{0, \dots, m\}$, and the result follows.

It remains to show that $\mathfrak{C}$ generates the Borel $\sigma$-field on $S^{\mathbb{N}}$. Clearly any $\bs{A}\in\mathfrak{C}$ is a Borel set in $S^\N$. For the reverse inclusion, by Lemma \ref{thm:product-topology} and the facts that $\sigma(\tau) = \sigma(\mathscr{B})$ and $\sigma(\tau^\N) = \sigma(\mathscr{B}^\N)$, it suffices to prove that given $\{A_i\}_{i =0}^{+\infty}$ a sequence of elements in $\mathscr{B}$, it holds that $\bs A = A_0 \times A_1 \times \dots \times A_n \times A_{n+1} \times \dots \in \sigma(\mathfrak{C})$. For each $m \in \mathbb{N}$, define 
$\bs A^{(m)} = A_0 \times \dots \times A_m \times S \times \dots$, i.e., $\bs A^{(m)} = \mathscr{C}(m; A_0, \dots, A_m)$. Surely, for each $m \in \mathbb{N}$, $\bs A^{(m)} \in \sigma(\mathfrak{C})$. Furthermore, note that $\bs A = \cap_{m=0}^{+\infty} \bs A^{(m)}$, so that $A \in \sigma(\mathfrak{C})$.
\end{proof}

\subsection{Additional auxiliary results}
\begin{definition}
Two measurable spaces $(M,\mathscr{M})$ and $(N,\mathscr{N})$ are said to be \emph{Borel isomophic} if there exists a bijection $h\colon M\to N$ such that both $h$ and $h^{-1}$ are measurable. A measurable space $(M,\mathscr{M})$ is said to be a \emph{Borel space} if it is Borel isomorphic to a Borel subset of the interval $[0,1]$.
\end{definition}

\begin{definition}
A topological space $M$ is said to be a \emph{Polish space} iff it is separable and admits a complete metrization.
\end{definition}

\begin{theorem}[\citep{kallenberg2002foundations}, Theorem~A1.2]
Let $M$ be a Polish space. Then every Borel subset of $M$ is a Borel space.
\end{theorem}

\begin{lemma}[Doob-Dynkin Lemma --- \citep{kallenberg2002foundations}, Lemma~1.13]\label{thm:doob-dynkin}
Let $(M,\mathscr{M})$ and $(N,\mathscr{N})$ be measurable spaces, and let $f\colon\Omega\to M$ and $g\colon \Omega\to N$ be any two given functions. If $M$ is Borel, then $f$ is $\sigma(g)$-measurable if and only if there exists a measurable mapping $h\colon N\to M$ such that $f = h\circ g$,
\end{lemma}

\begin{theorem}[Riesz-Markov]\label{thm:riesz-markov}
Let $S$ be a locally compact Hausdorff space and $\varphi$ a positive linear functional on $C_{\mathrm{c}}(S)$. Then there is a unique Radon measure $\mu$ on the Borel $\sigma$-field of $S$ for which
\(
\varphi(f) = \int_S f(x)\,\mu(\dd x)
\)
for all $f\in C_{\mathrm{c}}(S)$. In particular, if $S$ is compact and $\varphi(1) = 1$, then $\mu$ is a probability measure.
\end{theorem}

\begin{theorem}[Kolmogorov's strong law of large numbers]\label{thm:kolmogorov-slln}
Let $\bs{X}\coloneqq\left(X_0,X_1,\dots\right)$ be an independent sequence of random variables such that $\sup_n\Var\left(X_n\right)<\infty$. Then it holds that
\[
\lim_{n\to\infty}n^{-1}\sum_{i=0}^{n-1}\left(X_i - \E X_i\right) = 0
\]
almost surely.
\end{theorem}




\begin{theorem}[range and inverse, Kuratowski --- \citep{kallenberg2002foundations}, Theorem~A1.3]\label{thm:kuratowski} Let $f$ be a measurable bijection between two Borel spaces $S$ and $T$. Then the inverse $f^{-1}$ is again measurable.
\end{theorem}

\end{appendix}

\section*{Acknowledgements}
The author Luísa Borsato is supported by grant 2018/21067-0, S\~ao Paulo Research Foundation (FAPESP).

The author Eduardo Horta wishes to thank MCTIC/CNPq (process number 438642/2018-0) for financial support.





\end{document}